\documentclass[11pt]{article}
\usepackage{amssymb}
\usepackage{amsmath}
\usepackage{amsthm}
\usepackage{color}
\usepackage{comment}
\usepackage{geometry}		
\usepackage{graphicx}

\makeatletter
	\@addtoreset{equation}{section}
\makeatother

\let\originalleft\left
\let\originalright\right
\renewcommand{\left}{\mathopen{}\mathclose\bgroup\originalleft}
\renewcommand{\right}{\aftergroup\egroup\originalright}

\begin{document}

\newcommand\cM{\mathcal{M}}
\newcommand\cN{\mathcal{N}}

\newtheorem{theorem}{Theorem}[section]
\newtheorem{corollary}[theorem]{Corollary}
\newtheorem{lemma}[theorem]{Lemma}
\newtheorem{proposition}[theorem]{Proposition}

\theoremstyle{definition}
\newtheorem{definition}{Definition}[section]
\newtheorem{example}[definition]{Example}

\theoremstyle{remark}
\newtheorem{remark}{Remark}[section]

\title{
Three forms of dimension reduction for border-collision bifurcations.
}
\author{
D.J.W.~Simpson\\\\
School of Mathematical and Computational Sciences\\
Massey University\\
Palmerston North, 4410\\
New Zealand
}
\maketitle


\begin{abstract}

For dynamical systems that switch between different modes of operation, parameter variation can cause periodic solutions to lose or acquire new switching events. When this causes the eigenvalues (stability multipliers) associated with the solution to change discontinuously, we show that if one eigenvalue remains continuous then all local invariant sets of the leading-order approximation to the system occur on a lower dimensional manifold. This allows us to analyse the dynamics with fewer variables, which is particularly helpful when the dynamics is chaotic. We compare this to two other codimension-two scenarios for which dimension reduction can be achieved.

\end{abstract}

\section{Introduction}
\label{sec:intro}

Bifurcations are critical parameter values at which the
behaviour of a dynamical system changes fundamentally.
There many types of bifurcations, and much theory for explaining the related dynamics \cite{Ku04}.
A period-doubling bifurcation, for example, occurs when an eigenvalue associated with periodic solution becomes $-1$,
and creates a period-doubled solution when a non-degeneracy condition holds.

The wide applicability of bifurcation theory relies heavily on dimension reduction.
This allows the theory for the related dynamics to be applied to systems of any number of dimensions.
This can be achieved when the pertinent dynamics occurs entirely on a low-dimensional manifold.
In the case of period-doubling bifurcations, this manifold is a centre manifold associated with the eigenvalue $-1$.

But dimension reduction is usually not possible for border-collision bifurcations (BCBs).
These occur in systems with switches, jumps, or thresholds,
when a periodic solution experiences a change to its pattern of switching events \cite{DiBu08}.
BCBs are framed most simply in terms of piecewise-smooth maps,
where the bifurcation occurs when a fixed point collides with a switching manifold on which the map is nonsmooth.
Recent studies where BCBs play a key role
include sleep-wake patterns \cite{AtPi22},
influenza outbreaks \cite{RoHi19b},
and threshold harvesting for fisheries \cite{HiLi19}.

This Letter treats BCBs for continuous maps on $\mathbb{R}^n$
where each piece of the map is smoothly extendable through the switching manifold.
The local dynamics are typically characterised
by two sets of $n$ eigenvalues, $S_L$ and $S_R$,
associated with the fixed point
on each side of the switching manifold \cite{Si16}.
These dynamics may be chaotic,
and this occurs in models of open economies \cite{CaJa06},
power converters \cite{DiGa98,YuBa98},
and mechanical systems with stick-slip friction \cite{DiKo03,SzOs08}.
More types of dynamical behaviour is possible with every additional dimension \cite{GlJe15},
so in general dimension reduction of BCBs is not possible.

Given the importance of dimension reduction to the success of bifurcation theory broadly,
it is instructive to explore the extent to which it can be applied to BCBs.
The purpose of this Letter is to highlight
three scenarios in which dimension reduction is possible:
$S_L$ or $S_R$ contain a zero eigenvalue,
an eigenvalue with unit modulus,
or they share an eigenvalue.
Each scenario occurs as a codimension-two BCB,
and dimension reduction can be used to help explain the local dynamics.
The last scenario does not seem to have been reported previously, outside examples with $n=2$ in \cite{GhMc24}.
It can be interpreted as the special case that one eigenvalue does not change discontinuously as the bifurcation is crossed.
Below we show that when this occurs continuity forces the piecewise-linear approximation to the map
to admit an invariant hyperplane containing all bounded invariant sets.

\section{Preliminaries}
\label{sec:backg}

Consider a map
\begin{equation}
f(x;\xi) = \begin{cases}
f_L(x;\xi), & h(x;\xi) < 0, \\
f_R(x;\xi), & h(x;\xi) \ge 0,
\end{cases}
\label{eq:f}
\end{equation}
where $f_L : \mathbb{R}^n \times \mathbb{R}^m \to \mathbb{R}^n$,
$f_R : \mathbb{R}^n \times \mathbb{R}^m \to \mathbb{R}^n$,
and $h : \mathbb{R}^n \times \mathbb{R}^m \to \mathbb{R}$ are smooth.
For a given parameter vector $\xi \in \mathbb{R}^m$,
we are interested in the behaviour of the state variable $x \in \mathbb{R}^n$
under repeated iterations $x \mapsto f(x;\xi)$.

We assume the equation $h(x;\xi) = 0$ defines a smooth codimension-one switching manifold.
We also assume $f$ is continuous,
that is $f_L(x;\xi) = f_R(x;\xi)$ at all points on the switching manifold.

BCBs occur when a fixed point of $f_L$ or $f_R$ collides with the switching manifold as $\xi$ is varied.
The bifurcation only affects the dynamics locally,
so it is helpful to Taylor expand $f_L$, $f_R$, and $h$ about the bifurcation,
and truncate to leading order.
The resulting approximation is piecewise-linear
and often captures all local dynamics, for details see \cite{Si16}.

To this end, let us now consider an arbitrary piecewise-linear continuous map
\begin{equation}
g(x) = \begin{cases}
A_L x + b \,, & c^{\sf T} x < 0, \\
A_R x + b \,, & c^{\sf T} x \ge 0,
\end{cases}
\label{eq:g}
\end{equation}
where $A_L, A_R \in \mathbb{R}^{n \times n}$ and $b, c \in \mathbb{R}^n$.
The switching manifold of $g$ is
\begin{equation}
\Sigma = \left\{ x \in \mathbb{R}^n \,\middle|\, c^{\sf T} x = 0 \right\},
\label{eq:Sigma}
\end{equation}
and the assumption of continuity implies
\begin{equation}
A_R = A_L + p c^{\sf T},
\label{eq:continuityConstraint}
\end{equation}
for some $p \in \mathbb{R}^n$.
We write $g_L(x) = A_L x + b$
and $g_R(x) = A_R x + b$ for the left and right pieces of $g$.

If $g$ satisfies an observability condition \cite{Si16,Di03},
it can be converted to the border-collision normal form (BCNF) via an affine change of coordinates.
In two dimensions the BCNF consists of \eqref{eq:g} with
\begin{equation}
\begin{aligned}
A_L &= \begin{bmatrix} \tau_L & 1 \\ -\delta_L & 0 \end{bmatrix}, & \quad
A_R &= \begin{bmatrix} \tau_R & 1 \\ -\delta_R & 0 \end{bmatrix}, \\
b &= \begin{bmatrix} 1 \\ 0 \end{bmatrix}, &
c &= \begin{bmatrix} 1 \\ 0 \end{bmatrix},
\end{aligned}
\label{eq:2dform}
\end{equation}
while in three dimensions it uses
\begin{equation}
\begin{aligned}
A_L &= \begin{bmatrix} \tau_L & 1 & 0 \\ -\sigma_L & 0 & 1 \\ \delta_L & 0 & 0 \end{bmatrix}, & \quad
A_R &= \begin{bmatrix} \tau_R & 1 & 0 \\ -\sigma_R & 0 & 1 \\ \delta_R & 0 & 0 \end{bmatrix}, \\
b &= \begin{bmatrix} 1 \\ 0 \\ 0 \end{bmatrix}, &
c &= \begin{bmatrix} 1 \\ 0 \\ 0 \end{bmatrix}.
\end{aligned}
\label{eq:3dform}
\end{equation}
In general the BCNF has $2 n$ parameters.
These are the entries in the first columns of $A_L$ and $A_R$,
and are the coefficients of the characteristic polynomials of $A_L$ and $A_R$.
Hence they are characterised by the eigenvalues of $A_L$ and $A_R$,
which correspond to the sets $S_L$ and $S_R$ for the fixed point on each side of the BCB.
This is because the coordinate changes required to produce the BCNF
impose similarity transforms on the Jacobian matrices,
so do not alter their eigenvalues.

\section{Shared eigenvalues}
\label{sec:shared}

Here we study the scenario that
an eigenvalue $\lambda \in \mathbb{R}$ is common to both $S_L$ and $S_R$.
We first provide two examples,
then show how dimension reduction can be achieved in general.

\begin{figure}[b!]
\begin{center}
\includegraphics[width=\textwidth]{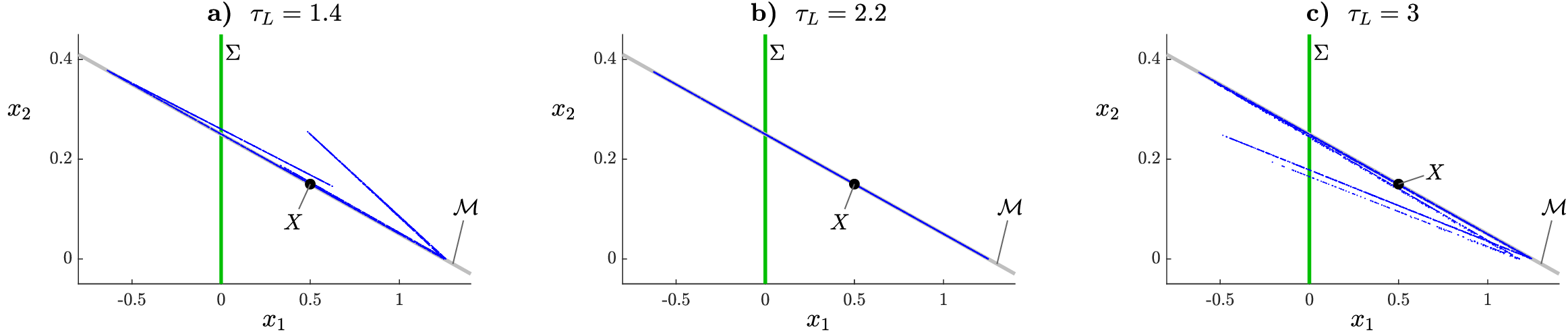}
\caption{
Phase portraits of \eqref{eq:g} with \eqref{eq:2dform}, \eqref{eq:2dparam},
and three different values of $\tau_L$.
The blue dots indicate the attractor,
$\cM$ is an invariant line for the right piece of the map,
and $\Sigma$ is the switching manifold.
\label{fig:QQ_a}
} 
\end{center}
\end{figure}

\begin{figure}[b!]
\begin{center}
\includegraphics[height=4cm]{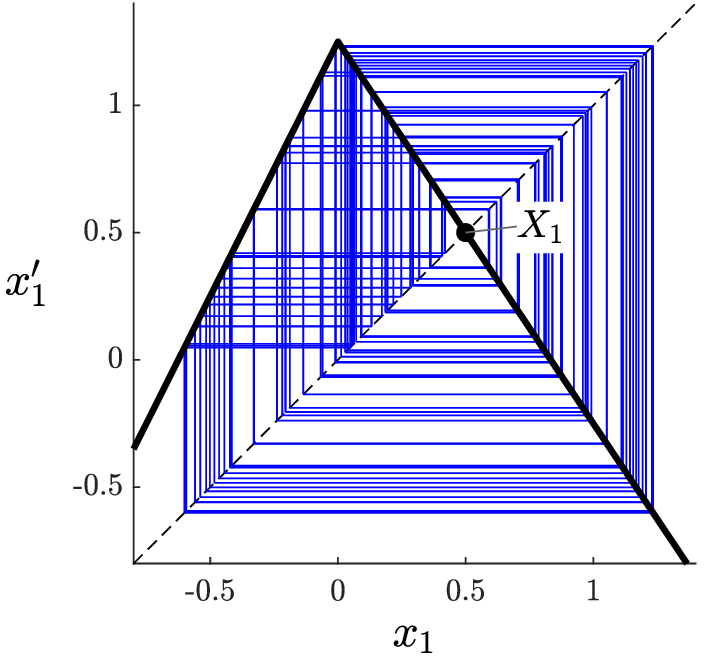}
\caption{
A cobweb diagram of the restriction of \eqref{eq:g} to $\cM$
for the parameter values of Fig.~\ref{fig:QQ_a}b.
This is a skew tent map with slopes $2$ and $-1.5$.
\label{fig:QQ_b}
} 
\end{center}
\end{figure}

Fig.~\ref{fig:QQ_a} shows phase portraits of the two-dimensional BCNF with
\begin{align}
\delta_L &= 0.4, &
\tau_R &= -1.3, &
\delta_R &= -0.3,
\label{eq:2dparam}
\end{align}
and three different values of $\tau_L$.
These represent the dynamics created by a family of BCBs parameterised by $\tau_L$.
The blue dots show $3000$ points of a typical forward orbit after transient dynamics have decayed,
and indicate the attractor of the map.
The black dot is the fixed point
\begin{equation}
X = (I - A_R)^{-1} b,
\label{eq:X}
\end{equation}
while the grey line is
\begin{equation}
\cM = \left\{ x \in \mathbb{R}^n \,\middle|\, u^{\sf T} (x - X) = 0 \right\},
\label{eq:M}
\end{equation}
where $u^{\sf T}$ is a left eigenvector of $A_R$ corresponding to the eigenvalue $\lambda = 0.2$.
This line is the affine subspace generated by the other eigenvalue of $A_R$,
and is invariant under the right piece of the map. 

The middle phase portrait uses $\tau_L = 2.2$ with which $\lambda$ is also an eigenvalue of $A_L$.
Here $\cM$ is also invariant under $g_L$,
and the restriction of $g$ to $\cM$ is a skew tent map, see Fig.~\ref{fig:QQ_b}.

\begin{figure}[b!]
\begin{center}
\includegraphics[width=\textwidth]{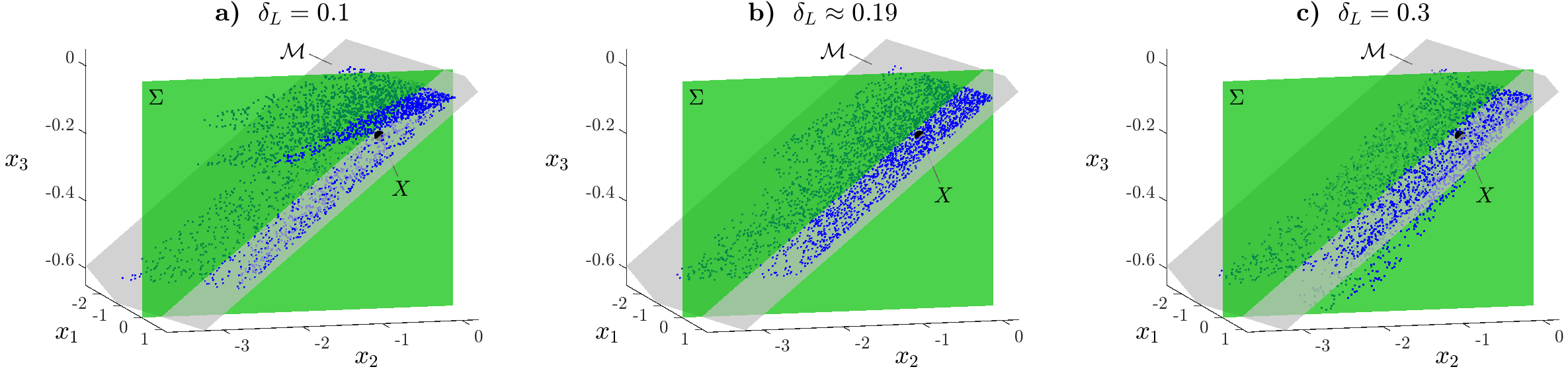}
\caption{
Phase portraits of \eqref{eq:g} with \eqref{eq:3dform}, \eqref{eq:3dparam},
and three different values of $\delta_L$.
In each plot the plane $\cM$ is invariant under the right piece of the map.
In panel (b), $\delta \approx 0.18973854$ is such that $A_L$ and $A_R$ have a common eigenvalue.
In this case $\cM$ is invariant under both pieces of the map.
\label{fig:QQ_c}
} 
\end{center}
\end{figure}

\begin{figure}[b!]
\begin{center}
\includegraphics[height=4cm]{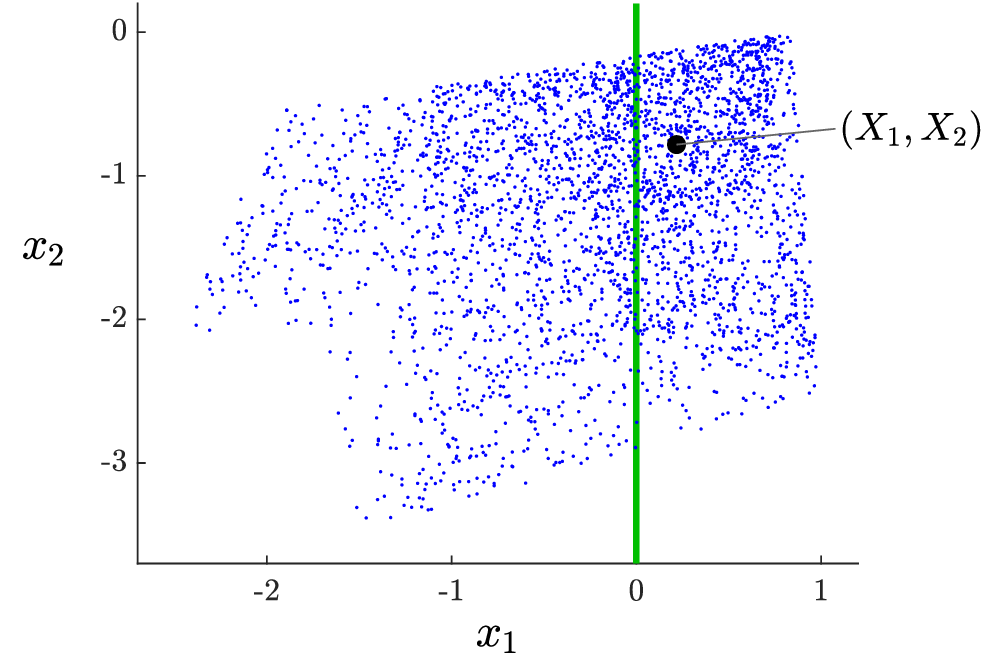}
\caption{
A phase portrait of the restriction of \eqref{eq:g} to $\cM$
for the parameter values of Fig.~\ref{fig:QQ_c}b.
\label{fig:QQ_d}
} 
\end{center}
\end{figure}

As another example, Fig.~\ref{fig:QQ_c} illustrates the dynamics of the three-dimensional BCNF with
\begin{align}
\tau_L &= 0, &
\sigma_L &= -1, &
\tau_R &= 0, &
\sigma_R &= 3, &
\delta_R &= -0.6,
\label{eq:3dparam}
\end{align}
and different values of $\delta_L$.
Again $X$ and $\cM$ are given by \eqref{eq:X} and \eqref{eq:M}
(now $\lambda \approx -0.1974$ is an eigenvalue of $A_R$ with corresponding left eigenvector $u^{\sf T}$).
The middle phase portrait uses $\delta_L$ such that $\lambda$ is also an eigenvalue of $A_L$.
Here the attractor belongs to $\cM$,
and Fig.~\ref{fig:QQ_d} shows a phase portrait of its restriction to $\cM$.

The above examples appear to be typical in that
similar dynamics is observed over wide ranges of parameter values.
The following result explains the above observations.

\begin{proposition}
Suppose $\lambda \in \mathbb{R}$ is an algebraic multiplicity one eigenvalue of $A_L$ and $A_R$.
Let $u^{\sf T}$ and $v$ be corresponding left and right eigenvectors of $A_R$.
Also suppose $\lambda \ne 1$ and $c^{\sf T} v \ne 0$.
Then for any $x \in \mathbb{R}^n$,
\begin{equation}
\phi(g(x)) = \lambda \phi(x),
\label{eq:distanceMap}
\end{equation}
where
\begin{equation}
\phi(x) = u^{\sf T} x - \frac{u^{\sf T} b}{1 - \lambda}.
\label{eq:phi}
\end{equation}
\label{pr:shared}
\end{proposition}

Proposition \ref{pr:shared} is proved in \ref{app:adjugates}.
If $1$ is not an eigenvalue of $A_R$,
then $X = (I - A_R)^{-1} b$ is well-defined
and $u^{\sf T} X = \frac{u^{\sf T} b}{1 - \lambda}$.
In this case $\cM$, defined by \eqref{eq:M}, is the set of points for which $\phi(x) = 0$.
By \eqref{eq:distanceMap}, $\cM$ is invariant under $g$.
Forward orbits converge to $\cM$ if $|\lambda| < 1$,
and diverge from $\cM$ if $|\lambda| > 1$.
In either case, any bounded invariant set of $g$ must belong to $\cM$.

\begin{remark}
The restriction of $g$ to $\cM$ is an $(n-1)$-dimensional piecewise-linear map.
In Fig.~\ref{fig:QQ_b} this map is one-dimensional
and its dynamics are well understood
(it has a single-interval chaotic attractor \cite{ItTa79}).
In Fig.~\ref{fig:QQ_d} the map is two-dimensional
and it may be possible to prove it has a two-dimensional attractor
via the methodology of Glendinning \cite{Gl16e}.
\end{remark}

\begin{remark}
For both examples the attractor appears to vary continuously with $\tau_L$
in Hausdorff metric \cite{GlSi20b}, and in measure \cite{AlPu17}.
Thus our understanding of the attractors obtained via dimension reduction
provides a broad explanation for the nature of the attractor at nearby values of the parameters.
\end{remark}

\begin{remark}
The manifold $\cM$ intersects $\Sigma$ transversally
unless $u$ is a scalar multiple of $c$.
This occurs only when $A_L$ and $A_R$ share {\em all} their eigenvalues, for the following reason.
Let $y$ be a right eigenvector corresponding to any eigenvalue $\nu \ne \lambda$ of $A_R$.
Since $u^{\sf T} y = 0$, if $u$ is a scalar multiple of $c$ then $c^{\sf T} y = 0$ and so
$A_L y = \nu y$ by \eqref{eq:continuityConstraint},
hence $\nu$ is also an eigenvalue of $A_L$.
\end{remark}

\section{Zero eigenvalues}
\label{sec:zero}

Now suppose $0$ is an eigenvalue of $A_L$ or $A_R$;
take $A_L$ without loss of generality.
Then the range of $g_L$ is an affine subspace $\cN \subset \mathbb{R}^n$
of dimension at most $n-1$.
In applications the presence of a zero eigenvalue is remarkably common \cite{Si24f},
in particular Poincar\'e maps of Filippov systems have a degenerate range
whenever one piece of the map corresponds to trajectories with segments of sliding motion \cite{DiKo03}.

Consider the most generic situation that $0$ is an algebraic multiplicity one eigenvalue of $A_L$,
and also suppose $1$ is not an eigenvalue of $A_L$.
Then $g_L$ has the unique fixed point
\begin{equation}
Y = (I - A_L)^{-1} b,
\label{eq:Y}
\end{equation}
and $\cN$ is $(n-1)$-dimensional, contains $Y$, and has directions given by right eigenvectors of all non-zero eigenvalues of $A_L$.
So $\cN$ is identical to $\cM$ of the previous section,
but uses the left piece of $g$ and $\lambda = 0$.

Forward orbits of $g$ either become constrained to the right half-plane $c^{\sf T} x \ge 0$,
where they are easy to understand because they are governed by the affine map $g_R$,
or repeatedly enter the left half-plane where they immediately map to $\cN$.
In the latter case we can obtain an $(n-1)$-dimensional characterisation of the dynamics via an induced map $G$
defined by the first return to $\cN$.
Specifically, for any $x \in \cN$ we define
\begin{equation}
G(x) = g^j(x),
\label{eq:inducedMap}
\end{equation}
for the smallest $j \ge 1$ for which $g^j(x) \in \cN$.

\begin{figure}[b!]
\begin{center}
\includegraphics[width=\textwidth]{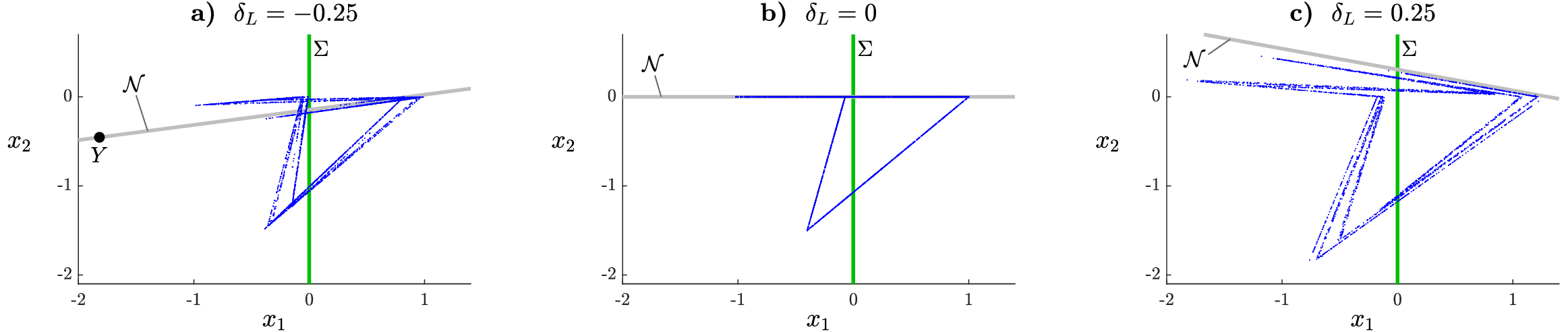}
\caption{
Phase portraits of \eqref{eq:g} with \eqref{eq:2dform}, \eqref{eq:2dparam2},
and three different values of $\delta_L$.
In each plot, $\cN$, defined by \eqref{eq:N}, is invariant under the left piece of the map.
In panel (b), every $x \in \mathbb{R}^2$ with $x_1 \le 0$ maps to $\cN$.
Note, in panels (b) and (c) the fixed point $Y$ lies outside the plotted range.
\label{fig:QQ_e}
} 
\end{center}
\end{figure}

\begin{figure}[b!]
\begin{center}
\includegraphics[height=4cm]{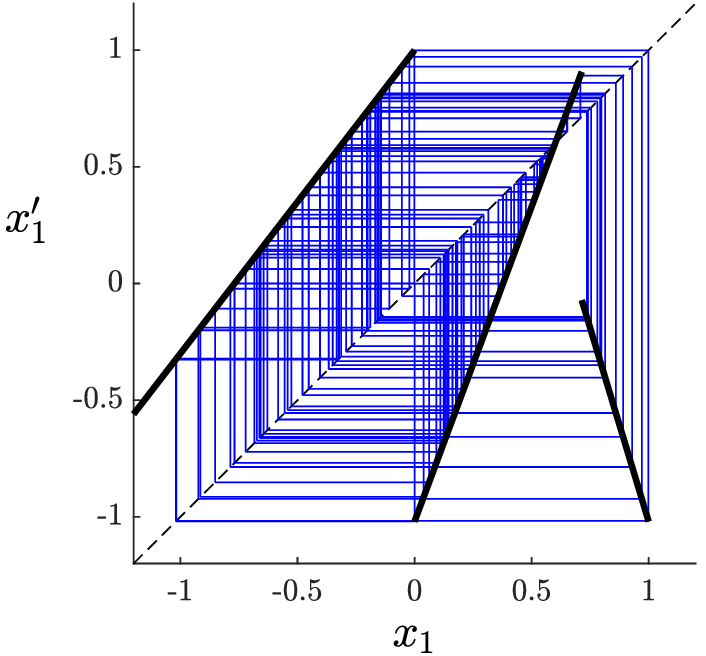}
\caption{
A cobweb diagram of the induced map $G$
(defined by the first return of orbits to $\cN$)
for the map shown in Fig.~\ref{fig:QQ_e}b.
Specifically, $\left( x_1', 0 \right) = G(x_1,0)$,
where $G$ is defined by \eqref{eq:inducedMap}.
\label{fig:QQ_f}
} 
\end{center}
\end{figure}

As an example, Fig.~\ref{fig:QQ_e} uses the two-dimensional BCNF with
\begin{align}
\tau_L &= 1.3, &
\tau_R &= -1.4, &
\delta_R &= 1.5,
\label{eq:2dparam2}
\end{align}
and three different values of $\delta_L$.
Each plot shows a typical forward orbit (with transient removed) and the line 
\begin{equation}
\cN = \left\{ x \in \mathbb{R}^n \,\middle|\, w^{\sf T} (x - Y) = 0 \right\},
\label{eq:N}
\end{equation}
where $w^{\sf T}$ is a left eigenvector of $A_L$ corresponding to its smallest eigenvalue in absolute value.
Notice $\delta_L$ is the determinant of $A_L$,
so $0$ is an eigenvalue of $A_L$ when $\delta_L = 0$,
and in this case the range of $g_L$ is given by \eqref{eq:N}.
For the map with $\delta_L = 0$, Fig.~\ref{fig:QQ_f} shows a cobweb diagram of the induced map $G$.
This map is discontinuous at points where the value of $j$ in \eqref{eq:inducedMap} changes.
Notice $G$ is piecewise-expanding (each piece has slope with absolute value greater than $1$),
so has a chaotic attractor by the results of Li and Yorke \cite{LiYo78}.
For further analysis of such dynamics refer to Kowalczyk \cite{Ko05}.

As another example Fig.~\ref{fig:QQ_g}, uses the three-dimensional BCNF with
\begin{align}
\tau_L &= 1.6, &
\sigma_L &= 0.8, &
\tau_R &= -1.5, &
\sigma_R &= 0, &
\delta_R &= 1,
\label{eq:3dparam2}
\end{align}
and three values of $\delta_L$.
Again, when $\delta_L = 0$ all points with $x_1 \le 0$ map to $\cN$.
Fig.~\ref{fig:QQ_h} shows the corresponding attractor of $G$.

As in the previous section both examples appear to involve robust chaos \cite{Gl17}.
The attractor varies continuously with parameters,
so our understanding of the attractor when $\delta_L = 0$ obtained using the induced map
extends to small $\delta_L \ne 0$.

\begin{figure}[h!]
\begin{center}
\includegraphics[width=\textwidth]{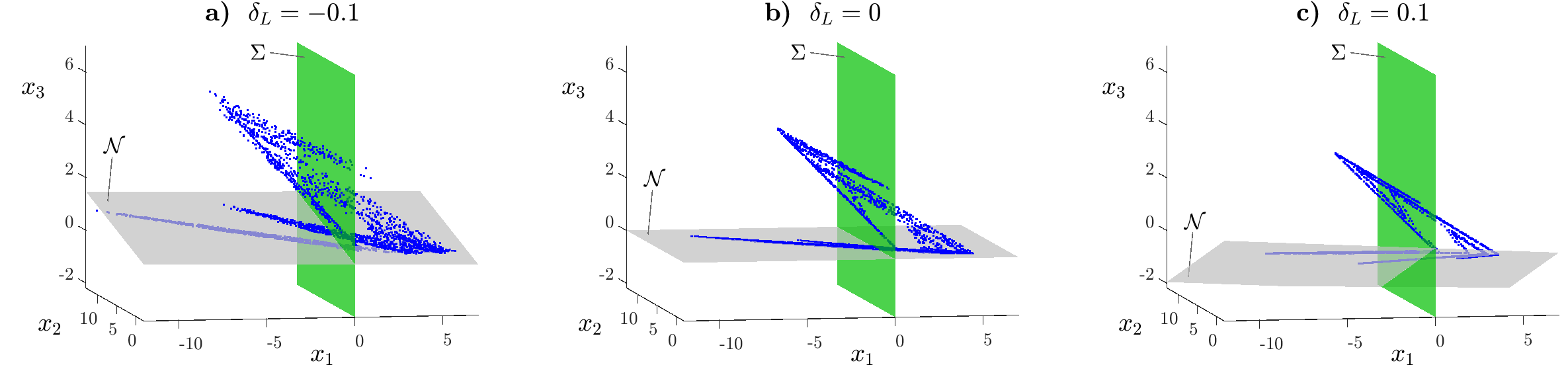}
\caption{
Phase portraits of \eqref{eq:g} with \eqref{eq:3dform}, \eqref{eq:3dparam2},
and three different values of $\delta_L$.
\label{fig:QQ_g}
} 
\end{center}
\end{figure}

\begin{figure}[h!]
\begin{center}
\includegraphics[height=4cm]{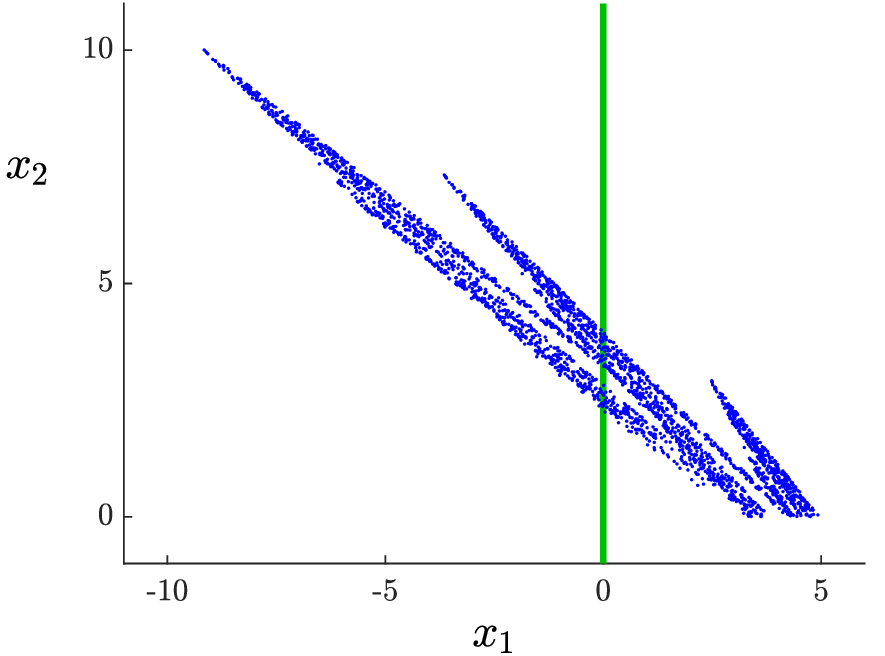}
\caption{
A phase portrait of the induced map $G$
for the map shown in Fig.~\ref{fig:QQ_g}b.
\label{fig:QQ_h}
} 
\end{center}
\end{figure}

\section{Modulus one eigenvalues}
\label{sec:nonhyperbolic}

Finally we review BCBs for which $A_L$ or $A_R$ has an eigenvalue with unit modulus,
for details see \cite{CoDe10,Si10}.
Here the dynamics of the piecewise-linear approximation is degenerate \cite{SuGa10},
so it is necessary to return to the general form \eqref{eq:f}.

Suppose a two-parameter family of maps \eqref{eq:f} admits a curve of BCBs.
Further suppose that at some point on this curve the fixed point of one piece of the map has an eigenvalue $1$.
In this case we can use a centre manifold of this piece of the map to perform dimension reduction
and obtain a one-dimensional description for some of its dynamics.
Indeed this analysis shows there exists a curve of saddle-node bifurcations
emanating from the BCB curve with a quadratic tangency, Fig.~\ref{fig:nonhyp}a.

If instead the fixed point of one piece of the map has an eigenvalue $-1$,
or eigenvalues ${\rm e}^{\pm {\rm i} \theta}$, where $0 < \theta < \pi$
and $\theta \ne \frac{2 \pi}{3}, \frac{\pi}{2}$,
then a curve of period-doubling or Neimark-Sacker bifurcations emanates from the BCB curve,
Fig.~\ref{fig:nonhyp}b,c.
Again we can study the dynamics on the centre manifold for one piece of the map,
and this shows that the period-doubled solution and invariant circle
collide with the switching manifold along curves
that are quadratically tangent to the period-doubling or Neimark-Sacker curves at the codimension-two points.
The subsequent change in dynamics brought out by these collisions
can be complex and remains to be fully explored \cite{Si16}.

\begin{figure}[h!]
\begin{center}
\includegraphics[width=8cm]{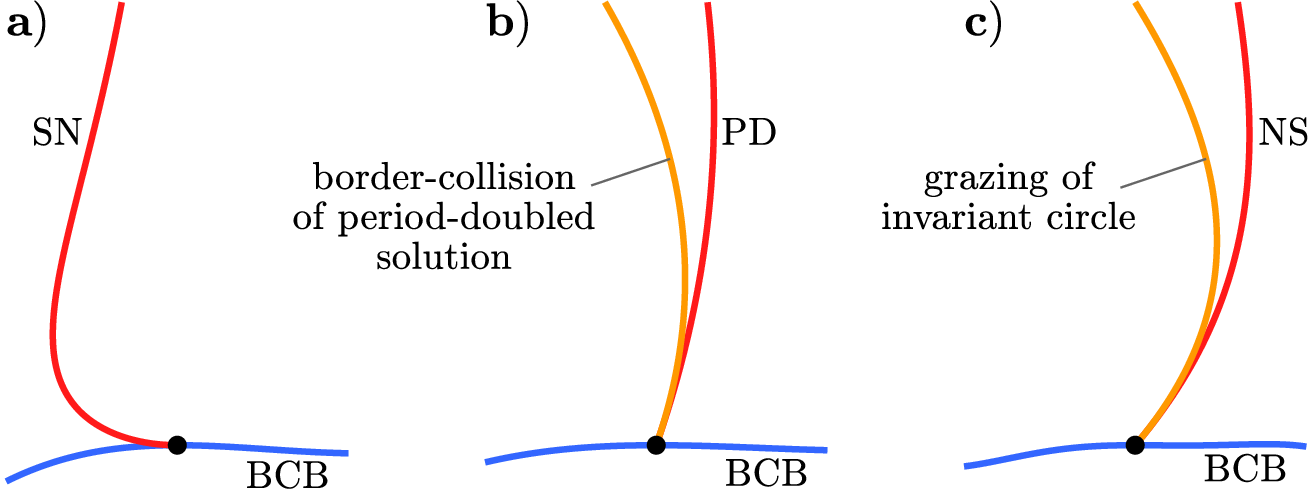}
\caption{
Sketches of two-parameter bifurcation diagrams for maps of the form \eqref{eq:f}
(BCB: border-collision bifurcation; SN: saddle-node bifurcation;
PD: period-doubling bifurcation; NS: Neimark-Sacker bifurcation).
\label{fig:nonhyp}
} 
\end{center}
\end{figure}

\section{Discussion}
\label{sec:conc}

For one-dimensional maps, the dynamics associated with BCBs are well understood \cite{GlSi24,NuYo95,SuAv16}.
For two-dimensional maps, the main features of these dynamics are well documented, e.g.~\cite{BaGr99},
while for higher dimensioanal maps the dynamics are challenging to catagorise systematically
because there are many parameters;
also the result of Glendinning and Jeffrey \cite{GlJe15}
shows that more dynamics is possible with every additional dimension.

However, many high-dimensional mathematical models with BCBs
contain symmetry, vastly differing time scales, or a related degeneracy.
By exploiting such features we hope to obtain a lower dimensional description of the dynamics
to which known theoretical results can be applied.

This Letter describes three codimension-two scenarios for which dimension reduction can be employed.
Two of the scenarios give an invariant codimension-one manifold.
At nearby parameter values there is no such invariant manifold,
which is a notable departure from normally hyperbolic manifolds of smooth maps that necessarily persist \cite{HiPu77}.
However, when the attracting objects vary continuously with parameters, as is often the case,
the reduction is still able to explain the overall nature of the long-term dynamics in open subsets of parameter space.

\section*{Acknowledgements}

This work was supported by Marsden Fund contract MAU2209 managed by Royal Society Te Ap\={a}rangi.

\appendix

\section{Adjugate matrices and calculations for shared eigenvalues}
\label{app:adjugates}

Here we use adjugate matrices to prove Proposition \ref{pr:shared}.
Most of the effort is in showing that the left eigenvector $u^{\sf T}$ of $A_R$, is also an eigenvector of $A_L$.

The {\em adjugate} of a matrix $A \in \mathbb{R}^{n \times n}$ is denoted ${\rm adj}(A)$
and is defined to have $(i,j)$-entry $(-1)^{i + j} m_{ji}$ for each $i,j = 1,\ldots,n$,
where $m_{ji}$ is the determinant of $A$ without its $i^{\rm th}$ column and $j^{\rm th}$ row.
The adjugate matrix obeys the identity,
\begin{equation}
{\rm adj}(A) A = A \,{\rm adj}(A) = \det(A) I,
\label{eq:adjId}
\end{equation}
so if $A$ is invertible then $A^{-1} = \frac{{\rm adj}(A)}{\det(A)}$.
The adjugate matrix appears in the matrix determinant lemma
\begin{equation}
\det \left( A + r q^{\sf T} \right) = \det(A) + q^{\sf T} {\rm adj}(A) r,
\label{eq:matrixDetLemma}
\end{equation}
which holds for any $A \in \mathbb{R}^{n \times n}$ and $q, r \in \mathbb{R}^n$.
To prove Proposition \ref{pr:shared} we use the following elementary result, see Sinkhorn \cite{Si93}.

\begin{lemma}
If $\lambda \in \mathbb{R}$ is an eigenvalue of $A \in \mathbb{R}^{n \times n}$ with algebraic multiplicity one,
then
\begin{equation}
{\rm adj}(\lambda I - A) = v u^{\sf T},
\label{eq:adjRankOne}
\end{equation}
for some non-zero vectors $u, v \in \mathbb{R}^n$.
\label{le:adjRankOne}
\end{lemma}

\begin{proof}[Proof of Proposition \ref{pr:shared}]
The vectors $u^{\sf T}$ and $v$ in \eqref{eq:adjRankOne} are necessarily left and right eigenvectors of $A$ corresponding to $\lambda$,
so we can assume the vectors $u$ and $v$ in Proposition \ref{pr:shared} have been scaled so that \eqref{eq:adjRankOne} holds.
Then by \eqref{eq:matrixDetLemma} and \eqref{eq:continuityConstraint},
\begin{equation}
\det(\lambda I - A_R) = \det(\lambda I - A_L) + c^{\sf T} {\rm adj}(\lambda I - A_L) p.
\nonumber
\end{equation}
But $\det(\lambda I - A_L) = \det(\lambda I - A_R) = 0$, so
\begin{equation}
0 = c^{\sf T} v u^{\sf T} p.
\nonumber
\end{equation}
Also $c^{\sf T} v \ne 0$, thus $u^{\sf T} p = 0$.
Hence, by \eqref{eq:continuityConstraint}, $u^{\sf T}$ is also a left eigenvector of $A_L$ corresponding to $\lambda$.

Now choose any $x \in \cM$ and let $Z = L$ if $c^{\sf T} x < 0$, and $Z = R$ otherwise.
Since $u^{\sf T} A_Z = \lambda u^{\sf T}$,
\begin{equation}
\phi(g(x)) = u^{\sf T} (A_Z x + b) - \frac{u^{\sf T} b}{1 - \lambda}
= \lambda u^{\sf T} x - \frac{\lambda u^{\sf T} b}{1 - \lambda}
= \lambda \phi(x),
\nonumber
\end{equation}
as required.
\end{proof}

{\footnotesize

}

\end{document}